\documentclass[12pt,oneside]{amsart}

\usepackage{amsthm,amsmath,amssymb}
\usepackage{mathrsfs}
\usepackage{enumerate}
\usepackage{mathtools}
\usepackage{multirow}
\usepackage{enumitem}
\usepackage[symbol]{footmisc}
\usepackage{amsfonts}
\usepackage{verbatim}
\usepackage{amsbsy}
\usepackage{amsmath}
\usepackage{amssymb}
\usepackage{MnSymbol}
\usepackage{mathrsfs}
\usepackage{cite}
\usepackage{algorithm}
\usepackage[noend]{algpseudocode}
\usepackage{float}
\usepackage{multicol}

\newtheorem{theorem}{Theorem}[section]
\newtheorem{lemma}[theorem]{Lemma}
\newtheorem{proposition}[theorem]{Proposition}
\newtheorem{conjecture}[theorem]{Conjecture}

\newtheorem{question}[theorem]{Question}

\theoremstyle{definition}
\newtheorem{definition}[theorem]{Definition}

\theoremstyle{remark}
\newtheorem{remark}[theorem]{Remark}
\newtheorem{example}[theorem]{Example}

\newtheorem{case2}{Case}

\numberwithin{subcase}{case}

\numberwithin{subsupposition}{supposition}

\numberwithin{subsupposition2}{supposition2}

\makeatletter
\def\@seccntformat#1{%
  \protect\textup{\protect\@secnumfont
    \ifnum\pdfstrcmp{subsection}{#1}=0 \bfseries\fi
    \csname the#1\endcsname
    \protect\@secnumpunct
  }%
}  
\makeatother

\makeatletter
\def\subsection{\@startsection{subsection}{3}%
  \z@{.7\linespacing\@plus.7\linespacing}{0.5\linespacing}
  {\normalfont\bfseries}}
\makeatother

\makeatletter
\@addtoreset{case}{theorem}
\makeatother

\makeatletter
\def\BState{\State\hskip-\ALG@thistlm}
\makeatother

\DeclareMathOperator{\gpr}{gpr}
\DeclareMathOperator{\cpr}{corepr}
\DeclareMathOperator{\mi}{mi}
\DeclareMathOperator{\pn}{Pn}


\begin{document}

\title{Parikh Motivated Study on Repetitions in Words}

\author{Ghajendran Poovanandran}
\address{School of Mathematical Sciences\\
Universiti Sains Malaysia\\
11800 USM, Malaysia}
\email{p.ghajendran@gmail.com}

\author{Adrian Atanasiu}
\address{Consulting Prof. at Faculty of Mathematics and Computer Science\\
Bucharest University\\
Str. Academiei 14\\
Bucharest 010014, Romania}
\email{aadrian@gmail.com}

\author{Wen Chean Teh}
\address{School of Mathematical Sciences\\
Universiti Sains Malaysia\\
11800 USM, Malaysia}
\email[Corresponding author]{dasmenteh@usm.my}

\begin{abstract}
We introduce the notion of general prints of a word, which is substantialized by certain canonical decompositions, to study repetition in words. These associated decompositions, when applied recursively on a word, result in what we term as core prints of the word. The length of the path to attain a core print of a general word is scrutinized. This \mbox{paper} also studies the class of square-free ternary words with respect to the Parikh \mbox{matrix} \mbox{mapping}, which is an extension of the classical Parikh mapping. It is shown that there are only finitely many matrix-equivalence classes of ternary words such that all words in each class are square-free. Finally, we \mbox{employ} \mbox{square-free} morphisms to generate infinitely many pairs of \mbox{square-free} ternary words that share the same Parikh matrix.
\end{abstract}

\maketitle

\section{Introduction}
The word $banana$ can be written as $ba(na)^2$.
Repetition in words is among the prominent themes that are studied in \mbox{combinatorics} on words and formal language theory. The systematic study of repetition in words dates back to the works of Axel Thue at the dawn of the $20^{\text{th}}$ century, where he first showed the existence of an infinitely long binary (respectively ternary) word that is cube-free (respectively square-free). In this paper, we explore new domains in studying repetition in words.

We newly introduce and study the notion of \textit{general prints} of word. A \mbox{general} print of a word is obtained by rewriting, in a certain prescribed fashion, repeated factors in a word into a single occurrence of that factor. Out of all the possible ways to reduce a given word in such manner, we choose two natural ones---that is to commence from the right or the left of that word. This paper also studies square-free (i.e.~repetition-free) words with respect to the notion of \textit{Parikh matrix mapping}. The latter, which was introduced in \cite{MSSY01}, is a generalization of the classical Parikh mapping \cite{rP66}. The Parikh matrix mapping is well-studied in the literature (for example, see \cite{GT17c,AAP08,wT14,MSY04,aS06,MBS17,GT17b}), particularly as a tool to deal with subword occurrences in words.

The remainder of this paper is structured as follows. Section 2 provides the basic terminology and preliminaries.
Section 3 introduces and scrutinize the left and right general prints of a word. Furthermore, the notion of core prints of a word is proposed and studied---a core print of a word is obtained by recursively reducing a word to its corresponding general print until the resulting word is no longer reducible. Section 4 presents new results on square-free words pertaining to the Parikh matrix mapping, exclusively for the ternary alphabet. It is shown that there are infinitely many pairs of square-free ternary words which share the same Parikh matrix. Our conclusions follow after that.

\section{Preliminaries}
The set of all positive integers is denoted by $\mathbb{Z}^+$ and let $\mathbb{N}=\mathbb{Z}^+\cup\{0\}$. The cardinality of a set $A$ is denoted by $|A|$.

Suppose $\Sigma$ is a finite and nonempty alphabet. The set of all words over $\Sigma$ is denoted by $\Sigma^*$ and $\lambda$ is the unique empty word. Let $\Sigma^+=\Sigma^*\backslash\{\lambda\}$. If $v,w\in\Sigma^*$, the
concatenation of $v$ and $w$ is denoted by $vw$. An ordered alphabet is an alphabet
$\Sigma=\{a_1,a_2,\ldots ,a_s\}$ with an ordering on it. For example, if $a_1<a_2<\cdots <a_s$, then we may write $\Sigma=\{a_1<a_2<\cdots <a_s\}$. For convenience, we shall frequently abuse notation and use $\Sigma$ to denote both the ordered alphabet and its underlying alphabet.

A word $v$ is a \emph{scattered subword} (or simply \textit{subword}) of $w\in \Sigma^*$ if and only if there exist $x_1,x_2,\dotsc, x_n$, $y_0, y_1, \dotsc,y_n\in \Sigma^*$ (possibly empty) such that \mbox{$v=x_1x_2\dotsm x_n \text{ and } w=y_0x_1y_1\dotsm y_{n-1}x_ny_n$}. If the letters in $v$ occur contiguously in $w$ (i.e.~$y_1=y_2=\dotsc=y_{n-1}=\lambda$), then $v$ is a \emph{factor} of $w$. A word $w\in\Sigma^*$
is \textit{square-free} iff it does not contain any factor of the form $u^2$ for some $u\in\Sigma^+$.

The number of occurrences of a word $v$ as a subword of $w$ is denoted by $\vert w\vert_v$. 
Two occurrences of $v$ are considered different if and only if they differ by at least one position of some letter. 
For example, $\vert abab\vert_{ab}=3$ and $\vert abcabc\vert_{abc}=4$.
By convention, $\vert w\vert_{\lambda}=1$ for all $w\in \Sigma^*$.

For any integer $n\geq 2$, let $\mathcal{M}_n$ denote the multiplicative monoid of $n\times n$ upper triangular matrices with nonnegative integral entries and unit diagonal.

\begin{definition} \cite{MSSY01}
Suppose $\Sigma=\{a_1<a_2<\cdots <a_k\}$ is an ordered alphabet. The \textit{Parikh matrix mapping} with respect to $\Sigma$, denoted by $\Psi_\Sigma$, is the morphism:
\begin{equation*}
\Psi_\Sigma:\Sigma^*\rightarrow\mathcal{M}_{k+1},
\end{equation*}
defined such that for every integer $1\le q\le k$, if $\Psi_\Sigma(a_q)=(m_{i,j})_{1\le i,j\le k+1}$, then
\begin{itemize}
\item $m_{i,i}=1$ for all $1\le i\le k+1$;
\item $m_{q,q+1}=1$; and 
\item all other entries of the matrix $\Psi_\Sigma(a_q)$ are zero.
\end{itemize} 
Matrices of the form $\Psi_\Sigma(w)$ for $w\in\Sigma^*$ are termed as  \textit{Parikh matrices}.
\end{definition}

\begin{theorem}\cite{MSSY01}\label{PropertiesParikhMat}
Suppose $\Sigma=\{a_1<a_2<\cdots <a_k\}$ is an ordered alphabet and $w\in\Sigma^*$. The matrix $\Psi_\Sigma(w)=(m_{i,j})_{1\le i,j\le k+1}$ has the following properties:
\begin{itemize}
\item $m_{i,j}=0$ for all $1\le j<i\le k+1$;
\item $m_{i,i}=1$ for all $1\le i\le k+1$;
\item $m_{i,j+1}=|w|_{a_ia_{i+1}\cdots a_j}$ for $1\le i\le j\le k$.
\end{itemize}
\end{theorem}

\begin{remark}\label{RemEntriesParMat}
Suppose $\Sigma=\{a<b<c\}$ and $w\in\Sigma^*$. Then
\begin{equation*}
\Psi_\Sigma(w)=\begin{pmatrix}
1 & |w|_a & |w|_{ab} & |w|_{abc} \\
0 & 1 & |w|_b & |w|_{bc}\\
0 & 0 & 1 & |w|_c\\
0 & 0 & 0 & 1
\end{pmatrix}.
\end{equation*}
\end{remark}

\begin{example}
Suppose $\Sigma=\{a<b<c\}$ and $w=abac$.
Then \begin{align*}
\Psi_{\Sigma}(w)&=\Psi_{\Sigma}(a)\Psi_{\Sigma}(b)\Psi_{\Sigma}(a)\Psi_{\Sigma}(c)\\
&=\begin{pmatrix}
1 & 1 & 0 &0 \\
0 & 1 & 0 & 0\\
0 & 0 & 1 & 0\\
0 & 0 & 0 & 1
\end{pmatrix}
\begin{pmatrix}
1 & 0 & 0 &0 \\
0 & 1 & 1 & 0\\
0 & 0 & 1 & 0\\
0 & 0 & 0 & 1
\end{pmatrix}
\begin{pmatrix}
1 & 1 & 0 &0 \\
0 & 1 & 0 & 0\\
0 & 0 & 1 & 0\\
0 & 0 & 0 & 1
\end{pmatrix}
\begin{pmatrix}
1 & 0 & 0 &0 \\
0 & 1 & 0 & 0\\
0 & 0 & 1 & 1\\
0 & 0 & 0 & 1
\end{pmatrix}
= \begin{pmatrix}
1 & 2 & 1 & 1 \\
0 & 1 & 1 & 1\\
0 & 0 & 1 & 1\\
0 & 0 & 0 & 1
\end{pmatrix}.
\end{align*}
\end{example}

\begin{definition}
Suppose $\Sigma$ is an ordered alphabet.
Two words $w,w'\in \Sigma^*$ are \emph{$M$\!-equivalent}, denoted by $w\equiv_Mw'$, iff $\Psi_{\Sigma}(w)=\Psi_{\Sigma}(w')$.
A word $w\in \Sigma^*$ is \emph{$M$\!-ambiguous} iff it is $M$\!-equivalent to another distinct word. Otherwise, $w$ is \emph{$M$\!-unambiguous}. For any word $w\in\Sigma^*$, we denote by $C_w$ the set of all words that are $M$\!-equivalent to $w$.
\end{definition}

The following rewriting rules, introduced in \cite{AAP08}, are elementary in deciding whether two words are $M$\!-equivalent. (The version provided here is stated exclusively for the ternary alphabet.) Suppose $\Sigma=\{a<b<c\}$ and $w,w'\in\Sigma^*$.
\vspace{0.2em}\begin{itemize}[leftmargin=1.8cm]
\item[Rule $E1$.] If $w=xacy$ and $w'=xcay$ for some $x,y\in\Sigma^*$, then $w\equiv_Mw'$.
\item[Rule $E2$.] If $w=x\alpha byb\alpha z$ and $w'=xb\alpha y\alpha bz$ for some $\alpha\in\{a,c\}$, $x,z\in\Sigma^*$ and $y\in\{\alpha,b\}^*$, then $w\equiv_Mw'$.
\end{itemize}\vspace{0.2em}

\begin{definition}\label{DefMEequivalence}
Suppose $\Sigma$ is an ordered alphabet and $w,w'\in\Sigma^*$.
\begin{enumerate}
\item We say that $w$ is {\it $1$-equivalent} to $w'$ if and only if $w'$ can be obtained from $w$ by finitely many applications of Rule $E1$.
\item We say that $w$ is {\it ME-equivalent}\footnote{The term elementary matrix equivalence (\textit{ME}-equivalence) was first introduced in \cite{aS10}.} to $w'$ if and only if $w'$ can be obtained from $w$ by finitely many applications of Rule $E1$ and $E2$.
\end{enumerate}
\end{definition}

\begin{example}
Suppose $\Sigma=\{a<b<c\}$. Consider
\begin{equation*}
w=ab\boldsymbol{ca}bcbbc\rightarrow aba\boldsymbol{cb}cb\boldsymbol{bc}\rightarrow ababccbcb=w'.
\end{equation*}
Thus $w$ is \textit{ME}-equivalent to $w'$.
\end{example}

The following notion, introduced by \c{S}erb\v{a}nu\c{t}\v{a} in \cite{SS06}, is closely related to the central object of study in the next section.
\begin{definition}\label{DefSerbPrint}
Suppose $\Sigma$ is an alphabet and $w\in \Sigma^*$. Suppose $w=a^{p_1}_1a^{p_2}_2\cdots a^{p_n}_n$ such that $a_i\in\Sigma$ and $p_i>0$ for all $1\le i\le n$ with $a_i\neq a_{i+1}$ for all $1\le i\le n-1$. The \textit{print} of $w$ is the word $a_1a_2\cdots a_n$.
\end{definition}

\section{General Prints of a Word}\label{SecGeneralPrint}

In this section, we introduce and study the notion of general prints of a word. We first present a canonical decomposition of words which will serve as a basis for our study. Note that this decomposition was first introduced in \cite{GT17d} as a means to obtain the \textit{(right) Parikh normal form} of a word. Hence we retain the notation used to denote the final product of this decomposition.

In this section, let $\Sigma$ be a fixed alphabet with size at least two.

\subsection{Parikh Normal Form}
\begin{definition}[Parikh Normal Form]\cite{GT17d}\label{RightLeftDecomp}
Suppose $w\in\Sigma^+$. 

\begin{itemize}[leftmargin=1em]
\item Define $R_w=\{\,(u,v,n)\in\Sigma^*\times\Sigma^+\times \mathbb{Z}^+\,\,|\,\,w=uv^n\,\}$.

\item Define $\tau_r(w)=\max\{\,n\in\mathbb{Z}^+\,\,|\,(u,v,n)\in R_w \text{ for some }u\in\Sigma^* \text{ and } v\in\Sigma^+\,\}$.

\item Define $\theta_r(w)$ as follows:
\begin{itemize}
\item if $\tau_r(w)=1$, then $\theta_r(w)$ is defined to be the minimum element of the following set:
$$\{\,|v|\,\,|\,\,v\in\Sigma^+ \text{ and }
w=uv \text{ for some } u\in\Sigma^+\text{ with }\tau_r(u)\neq 1\},$$
provided it is nonempty; otherwise $\theta_r(w)=|w|$.
\item if $\tau_r(w)>1$, then $\theta_r(w)$ is defined to be the maximum element of the following set:
$$\{\,|v|\,\,|\,\,v\in\Sigma^+ \text{ and }
w=uv^{\tau_r(w)}\text{ for some } u\in\Sigma^+\}.$$
\end{itemize}

\item Define $\rho_r(w)=(u',v',\tau_r(w))$ to be the unique triplet in $R_w$ such that $|v'|=\theta_r(w)$.

\item Let $w_0=w$ and $(w_1,v_0,n_0)=\rho_r(w_0)$. For all integers $i\ge 1$ and while $w_i\neq \lambda$, recursively define $(w_{i+1},v_i,n_i)=\rho_r(w_i)$. Let $k\ge 0$ be the largest integer such that $w_k\neq \lambda$. 
\end{itemize}
We denote the form $v_{k}^{n_{k}}v_{k-1}^{n_{k-1}}\cdots v_0^{n_0}$ of $w$ by $\pn_r(w)$.






\end{definition}

\begin{remark}
The requirement $v\in\Sigma^+$ in the first item of Definition~\ref{RightLeftDecomp} eliminates
the trivial decomposition of a word $w$ into $w=w\lambda^n$ at each stage as $n$ does
not have an upper bound in this case.
\end{remark}


The following example illustrates the mechanism of the right decomposition of a word.

\begin{example}\label{ExRightDecomp}
Suppose $\Sigma=\{a,b,c\}$. Consider the word $w=abcccbcabbabb$.

Starting from right to left, we first aim to decompose the word $w$ such that the power of the right most component is the highest. In this case, the highest such power is two, where $abcccbcabbabb$ can be decomposed to either $abcccbcabbab^2$ or $abcccbc(abb)^2$. Since $|abb|>|b|$, our choice of decomposition would be $abcccbc(abb)^2$.

Next, we look at the remaining part of $w$, which is yet to be decomposed, that is $abcccbc$. Here, the highest power attainable on the right most component is one. Therefore, we decompose $abcccbc$ in a way that the right most component has the shortest length such that the remaining part can be decomposed to a power higher than one. That is to say, we decompose $abcccbc$ into $abccc(bc)^1$.

The part that remains to be decomposed now is $abccc$. Continuing the process as in the first step, we decompose $abccc$ into $abc^3$.

The final remaining part of $w$ is $ab$. The highest power attainable on the right most component here is one. Furthermore, there is no way for us to decompose $ab$ such that there exists some remaining factor which can be decomposed to a power higher than one. Thus the final component is $(ab)^1$.

Therefore, we have $\pn_r(w)=(ab)^1c^3(bc)^1(abb)^2$. Omitting the parentheses and power when the latter is one, we write $\pn_r(w)=abc^3bc(abb)^2$.
\end{example}

We define the left decomposition of a word $w$ analogously to Definition~\ref{RightLeftDecomp} such that the decomposition commences from left to right. Furthermore, in a similar fashion, we denote the final product of the left decomposition of $w$ by $\pn_l(w)$.  

\begin{example}
Suppose $\Sigma=\{a,b,c\}$. Consider the word $w=abababacbcbc$.
Then $\pn_l(w)=(ab)^3a(cb)^2c$.
\end{example}

The following result holds directly by the definitions of the right and left decompositions of a word.

\begin{proposition}\label{MirrorImageProperty}
Suppose $w\in\Sigma^+$. Let $$\pn_r(w)=v_{k}^{n_{k}}\cdots v_1^{n_1}v_0^{n_0} \text{ and } \pn_l(w)={u}_0^{m_0}{u}_{1}^{m_{1}}\cdots {u}_j^{m_j}$$ 
for some integers $j,k\in\mathbb{N}$, $m_i\in\mathbb{Z}^+\,(0\le i\le j)$ and $n_i\in\mathbb{Z}^+\,(0\le i\le k)$, and words $u_i\in\Sigma^+\,(0\le i\le j)$ and $v_i\in\Sigma^+\,(0\le i\le k)$. 
Then 
\begin{itemize}
\item $\pn_r(\mi(w))=[\mi({u}_j)]^{m_j}\cdots [\mi({u}_{1})]^{m_1}[\mi({u}_{0})]^{m_0}$;
\item $\pn_l(\mi(w))=[\mi({v}_0)]^{n_0}[\mi({v}_1)]^{n_1}\cdots [\mi({v}_{k})]^{n_k}$.
\end{itemize}
\end{proposition}

\subsection{General Prints}

\begin{definition}[General Prints]\label{DefGeneralPr}
Suppose $w\in\Sigma^+$. Let $$\pn_r(w)=v_{k}^{n_{k}}\cdots v_1^{n_1}v_0^{n_0} \text{ and } \pn_l(w)={u}_0^{m_0}{u}_{1}^{m_{1}}\cdots {u}_j^{m_j}$$ 
for some integers $j,k\in\mathbb{N}$, $m_i\in\mathbb{Z}^+\,(0\le i\le j)$ and $n_i\in\mathbb{Z}^+\,(0\le i\le k)$, and words $u_i\in\Sigma^+\,(0\le i\le j)$ and $v_i\in\Sigma^+\,(0\le i\le k)$. 
\begin{enumerate}
\item The \textit{right general print} of $w$, denoted by $\gpr_R(w)$, is the word $v_{k}\cdots v_1v_0$.
\item The \textit{left general print} of $w$, denoted by $\gpr_L(w)$, is the word $u_0u_1\cdots u_j$. 
\end{enumerate}
\end{definition}

\begin{remark}\label{RemGeneralPrint}
Express a word $w\in\Sigma^+$ in the form $y^{n_1}_1y^{n_2}_2\cdots y^{n_k}_k$ such that $y_i\in\Sigma^*$ and $n_i>0$ for all $1\le i\le k$ with $y_i\neq y_{i+1}$ for all $1\le i\le k-1$. Informally, the word $y_1y_2\cdots y_k$ can be regarded as a general print (associated to the decomposition) of $w$ . However, in this paper, we study only the ones as in Definition~\ref{DefGeneralPr} as they are the two natural canonical forms. The notion of a general print of a word is in fact a generalization of the notion of the print (see Definition~\ref{DefSerbPrint}) of a word.
\end{remark}

\begin{example}
Suppose $\Sigma=\{a,b,c\}$ and consider the word $w=cabccabc$. Then $\pn_r(w)=\pn_l(w)=(cabc)^2$. Thus both the left and right general prints of $w$ are the same, which is $cabc$. On the other hand, the (\c{S}erb\v{a}nu\c{t}\v{a}'s) print of $w$ is $cabcabc$.
\end{example}

The following assertion holds by the definitions and some simple observation.
\begin{remark}
For any word $w\in\Sigma^+$, if $\mi(w)=w$, then $\gpr_R(w)=\gpr_L(w)$.
\end{remark}

\begin{proposition}\label{PropMirrorGpr}
For every $w\in\Sigma^+$, we have
\begin{enumerate}
\item $\gpr_R(w)=\mi(\gpr_L(\mi(w)))$;
\item $\gpr_L(w)=\mi(\gpr_R(\mi(w)))$.
\end{enumerate}
\end{proposition}
\begin{proof}
It suffices to prove $(1)$ as $(2)$ follows immediately from $(1)$.

Let $\pn_r(w)=v_{k}^{n_{k}}\cdots v_1^{n_1}v_0^{n_0}$ for some integers $k\in\mathbb{N}$, $n_i\in\mathbb{Z}^+\,(0\le i\le k)$ and words \mbox{$v_i\in\Sigma^+\,(0\le i\le k)$.} Then, by Proposition~\ref{MirrorImageProperty}, we have 
$\pn_l(\mi(w))=[\mi({v}_0)]^{n_0}[\mi({v}_1)]^{n_1}\cdots [\mi({v}_{k})]^{n_k}$.
Correspondingly, we have $$\gpr_R(w)=v_k\cdots v_1v_0 \text{ and } \gpr_L(\mi(w))=\mi({v}_0)\mi({v}_1)\cdots \mi({v}_k).$$
It remains to see that 
\begin{center}
\begin{tabular}{lcl}
$\mi(\gpr_L(\mi(w)))$&$=$&$\mi(\mi(v_0)\mi(v_1)\cdots \mi(v_k))$\\
&$=$&$\mi(\mi(v_0))\mi(\mi(v_1))\cdots \mi(\mi(v_k))$\\
&$=$&$v_k\cdots v_1v_0$\\
&$=$&$\gpr_R(w)$.
\end{tabular}
\end{center}
\end{proof}

The following shows that for a general alphabet, in the case where the right and left general prints of a word are different, the respective lengths of the general prints can either be the same or different.

\begin{example}
Suppose $a,b\in\Sigma$. Consider the words $w=babaabaa$ and $w'=babaa$. Then 
\begin{itemize}
\item $\pn_r(w)=ba(baa)^2$ and $\pn_l(w)=(ba)^2aba^2$, hence $\gpr_R(w)=babaa$ and $\gpr_L(w)=baaba$.
\item $\pn_r(w')=baba^2$ and $\pn_l(w')=(ba)^2a$, hence $\gpr_R(w')=baba$ and $\gpr_L(w')=baa$.
\end{itemize}
\end{example}

\begin{question}
Out of the all the possible general prints (see Remark~\ref{RemGeneralPrint}) of a word $w\in\Sigma^+$, does $\gpr_R(w)$ or $\gpr_L(w)$ give you the one of the shortest length?
\end{question}

The answer is no. Suppose $a,b,c\in\Sigma$. Consider the word $w=abcbcbcabcbcbca$. We have $\pn_r(w)=a(bcbcbca)^2$ and  $\pn_l(w)=(abcbcbc)^2a$. Therefore $\gpr_R(w)=\gpr_L(w)=abcbcbca$. However, notice that another possible decomposition of $w$ is $a(bc)^3a(bc)^3a$. This gives a shorter general print of $w$, which is $abcabca$. 

The shortest general print of a word is however not necessarily unique. For instance, consider the word $w=abababcbcbc$. We have $\pn_r(w)=a(ba)^2(bc)^3$ and $\pn_l(w)=(ab)^3(cb)^2c$, hence $\gpr_R(w)=ababc$ and $\gpr_L(w)=abcbc$. It can be easily verified that both $\gpr_R(w)$ and $\gpr_L(w)$ are the shortest general prints of $w$.

\subsection{Core Prints}

\begin{definition}[Core Prints]\label{DefCorePrint}
Suppose $w\in\Sigma^+$. Let $w_0=w'_0=w$. For all integers $i\ge 0$, recursively define $w_{i+1}=\gpr_R(w_i)$ and $w'_{i+1}=\gpr_L(w'_i)$. Let $I$ (respectively $I'$) be the least nonnegative integer such that $w_{I}=w_{I+1}$ (respectively $w'_{I'}=w'_{I'+1}$).
\begin{enumerate}[leftmargin=2em]
\item The \textit{right core print} of $w$, denoted by $\cpr_R(w)$, is the word $w_{I}$.
\item The \textit{left core print} of $w$, denoted by $\cpr_L(w)$, is the word $w'_{I'}$.
\end{enumerate}
Let $l_R(w)$ and $l_L(w)$ denote the integers $I$ and $I'$ respectively.
\end{definition}

\newpage

\begin{remark}\label{SqFreeGenPrint}
For every $w\in\Sigma^+$, the following are equivalent:
\begin{itemize}
\item $\gpr_R(w)=w$;
\item $\gpr_L(w)=w$;
\item $\cpr_R(w)=w$;
\item $\cpr_L(w)=w$;
\item $w$ is square-free.
\end{itemize}
\end{remark}

\begin{remark}\label{CorePrBinary}
Suppose $\Sigma=\{a,b\}$ and $w\in\Sigma^+$. Then $\cpr_R(w),\cpr_L(w)\in\{a,b,ab,ba,aba,bab\}$.
\end{remark}

\begin{theorem}\label{CorePrintSame}
Suppose $|\Sigma|=2$. For every $w\in\Sigma^+$, we have $\cpr_R(w)=\cpr_L(w)$. 
\end{theorem}
\begin{proof}
Let $\Sigma=\{a<b\}$. If $w$ is either $a^p$ or $b^p$ for some positive integers $p$, then the conclusion trivially holds. 

Assume $w=ayb$ for some $y\in\Sigma^*$. Then, both $\cpr_R(w)$ and $\cpr_L(w)$ must start with a letter $a$ and end with a letter $b$. By Remark~\ref{CorePrBinary}, this is only possible if $\cpr_R(w)=ab=\cpr_L(w)$. By similar argument, it can be shown that if $w=bya$ for some $y\in\Sigma^*$, then $\cpr_R(w)=ba=\cpr_L(w)$.

Assume $w=aya$ for some $y\in\Sigma^*$ with $|y|_b\ge 1$. Then, both $\cpr_R(w)$ and $\cpr_L(w)$ must start and end with a letter $a$ and contain at least one letter $b$ in between. By Remark~\ref{CorePrBinary}, this is only possible if $\cpr_R(w)=aba=\cpr_L(w)$. By similar argument, it can be shown that if $w=byb$ for some $y\in\Sigma^*$ with $|y|_a\ge 1$, then $\cpr_R(w)=bab=\cpr_L(w)$. Thus our conclusion holds.
\end{proof}

Theorem~\ref{CorePrintSame} however cannot be extended to cater for larger alphabets, as illustrated in the following example.
\begin{example}\label{ExampleCoreDiff}
Suppose $a,b,c\in\Sigma$. Consider the word $w=ababcbabc$. 
We have
\begin{itemize}
\item \begin{itemize}
\item $\pn_r(w)=a(babc)^2$, thus $w_1=\gpr_R(w)=ababc$;
\item $\pn_r(w_1)=(ab)^2c$, thus $w_2=\gpr_R(w_1)=abc$;
\item $\pn_r(w_2)=abc$, thus $w_3=\gpr_R(w_2)=abc$.
\end{itemize}
\item \begin{itemize}
\item $\pn_l(w)=(ab)^2cbabc$, thus $w_1=\gpr_L(w)=abcbabc$;
\item $\pn_l(w_1)=abcbabc$, thus $w_2=\gpr_L(w_1)=abcbabc$;
\end{itemize}
\end{itemize}
Therefore, $\cpr_R(w)=abc\neq abcbabc=\cpr_L(w)$.
\end{example}
The word in Example~\ref{ExampleCoreDiff}, which is of length nine, is in fact a counterexample of the shortest length. The other such words are listed below:

\begin{center}
\begin{tabular}{ c c c c }
$cbcbabcba$ & $abcbabcbc$ & $acacbcacb$ & $acbcacbcb$\\ 
$babacabac$ & $bacabacac$ & $bcacbcaca$ & $bcbcacbca$\\  
$cabacabab$ & $cacabacab$ & $cbabcbaba$ &
\end{tabular}
\end{center}
Interestingly, there are only 12 such words out of $3^9=19683$ ternary words of length $9$.

Every word over $\Sigma$ corresponds to a unique \mbox{sequence} of decompositions to attain the right (respectively left) core print of that word. We now introduce a function that \mbox{captures}, for every positive integer $n$, the maximal length of such sequences with respect to the set of all words over $\Sigma$ with length $n$.

\begin{definition}\label{DefCharFunct}
Suppose $r\ge 2$ is an integer. The \textit{core print characteristic function} of order $r$ is the function $\zeta_r:\mathbb{Z}^+\rightarrow\mathbb{N}$ defined as
\begin{equation*}
\zeta_r(n)=\max\{k\in\mathbb{N}\,\,|\,\,l_R(w)=k\text{ for some }w\in\Sigma^*\text{ with }|w|=n\}
\end{equation*}
where $\Sigma$ is any alphabet with $|\Sigma|=r$.
\end{definition}

\begin{remark}
In general, for a word $w$, the values of $l_R(w)$ and $l_L(w)$ may not be the same.
However, by some simple analysis and Proposition~\ref{PropMirrorGpr}, one could see that $l_R(w)=l_L(\mi(w))$ for any $w\in\Sigma^*$. Thus changing the condition $l_R(w)=k$ in the definition of $\zeta_r(n)$ to $l_L(w)=k$ does not alter the function. The current choice is simply a matter of preference.
\end{remark}

Appendix~\ref{A1} exhausts the values of $\zeta_2(n)$ for every integer $1\le n\le 30$. 
The following are the (only) words $w\in\{a,b\}^*$ with length 30 such that $l_R(w)=6=\zeta_2(30)$. (Meanwhile, there are 25924760 words $w$ with length 30 such that $l_R(w)=5$.)

\vspace{0.7em}\begin{center}
        \begin{tabular}{ll}
abaababaaabaaabaabbaabbbaabbab, & abaababaaabaababbaaabbbaaabbab,\\
abaababaaababaabbaaabbbaaabbab, & abaababaaabbaaabbaabaabbaabbab,\\
abaababaaabbaaabbaabaabbabbaab, & abaababaabbaaaabbaabaabbaabbab,\\
abaababaabbaaaabbaabaabbabbaab, & abaababaabbaaabaabbaabbbaabbab,\\
abaababaabbaababbaaabbbaaabbab, & babbababbbabbbabbaabbaaabbaaba,\\
babbababbbabbabaabbbaaabbbaaba, & babbababbbababbaabbbaaabbbaaba,\\
babbababbbaabbbaabbabbaabbaaba, & babbababbbaabbbaabbabbaabaabba,\\
babbababbaabbbbaabbabbaabbaaba, & babbababbaabbbbaabbabbaabaabba,\\
babbababbaabbbabbaabbaaabbaaba, & babbababbaabbabaabbbaaabbbaaba.
        \end{tabular}
\end{center}

For instance, one can see that for the first word in the above list, the path to attain its right core print is as follows:
\begin{example}
Let $w_0=abaababaaabaaabaabbaabbbaabbab$. We have
\begin{itemize}
\item $\pn_r(w_0)=aba(ab)^2a^3ba(aab)^2(baabb)^2ab$, thus $w_1=\gpr_R(w_0)$ \\$=abaababaaabbaabbab$;
\item $\pn_r(w_1)=aba(ab)^2a^2(abba)^2b$, thus $w_2=\gpr_R(w_1)=abaabaabbab$;
\item $\pn_r(w_2)=a(baa)^2b^2ab$, thus $w_3=\gpr_R(w_2)=abaabab$;
\item $\pn_r(w_3)=aba(ab)^2$, thus $w_4=\gpr_R(w_3)=abaab$;
\item $\pn_r(w_4)=aba^2b$, thus $w_5=\gpr_R(w_3)=abab$;
\item $\pn_r(w_5)=(ab)^2$, thus $w_6=\gpr_R(w_3)=ab$.
\item $\pn_r(w_6)=ab$, thus $w_7=\gpr_R(w_6)=ab$.
\end{itemize}
\end{example}
 
\begin{remark}
For all integers $r\ge 2$ and $n\ge 1$, we have $\zeta_{r+1}(n)\ge \zeta_r(n)$. 
\end{remark}
For the case of $\zeta_3(n)$, we have computationally checked that $\zeta_3(n)=\zeta_2(n)$ for every integer $1\le n\le 14$ but $5=\zeta_3(15)\neq\zeta_2(15)=4$. An example of a word $w\in\{a,b,c\}^*$ with length 15 such that $l_R(w)=5$ is $cbaccaacacaacba$.

Appendix~\ref{A1} also suggests the possibility that the function $\zeta_2$ is nondecreasing. In general, appending a letter to the right or left of a word $w$ may reduce the value of $l_R(w)$. The following shows an extreme-case example of this.

\begin{example}
Suppose $\Sigma=\{a,b\}$ and consider the word $w=abaabbabbbabb$ of length 13. We have $l_R(w)=4=\zeta_2(13)$. However, $l_R(aw)=l_R(bw)=l_R(wa)=3$ and $l_R(wb)=2$.
\end{example}
However, if the following more general assertion holds, then the monotonicity of the function $\zeta_2$ is implied directly.
\begin{conjecture}\label{ConjecturePath}
For any word $w\in\Sigma^+$, a letter $x\in\Sigma$ can be inserted into $w$ to obtain a word $w'$ such that $l_R(w')\ge l_R(w)$.
\end{conjecture}

\section{On Square-free Words and \textit{M}-equivalence over the Ternary Alphabet}

A notion often investigated when dealing with the subject of repetition in words is square-freeness. In Section~\ref{SecGeneralPrint}, we see that the square-free property of a word has direct consequences on the general prints and core prints of that word (see Remark~\ref{SqFreeGenPrint}).

We now investigate the class of square-free words over the ternary alphabet with respect to the Parikh matrix mapping. In particular, we present new results on square-free ternary words pertaining to the notion of $M$\!-equivalence. 

\begin{lemma}\label{LemmaStructureSquareFree}
Suppose $\Sigma=\{a<b<c\}$ and $w\in\Sigma^*$ with $|w|_b=k\ge 3$. Assume every word in $C_w$ is square free. Write $w=u_0bu_1bu_2\cdots bu_k$ where $u_i\in\{a,c\}^*$ for every integer $0\le i\le k$. Then, $u_i\in\{a,c\}$ for every integer $1\le i\le k-1$ and $u_0,u_k\in\{\lambda,a,c\}$ such that $u_i\neq u_{i+1}$ for every integer $0\le i\le k-1$.
\end{lemma}
\begin{proof}
Since every word in $C_w$ is square-free (by the hypothesis), clearly $w$ is square-free as well. To show that $u_i\neq u_{i+1}$ for every integer $0\le i\le k-1$, we argue by contradiction. Assume $u_m=u_{m+1}$ for some integer $0\le m\le k-1$. Notice that if $0\le m\le k-2$, then $$w=yu_mbu_{m+1}by'=yu_mbu_mby'=y(u_mb)^2y'$$ for some $y,y'\in\Sigma^*$. On the other hand, if $m=k-1$, then $$w=u_0bu_1\cdots bu_{k-1}bu_k=u_0bu_1\cdots bu_kbu_k=u_0bu_1\cdots (bu_k)^2.$$ In both cases, we have a contradiction as $w$ is square-free.

To prove the remaining part of the assertion, note that for every integer $1\le i\le k-1$, the word $u_i$ has to be nonempty. Otherwise, the square $b^2$ will be a factor in $w$. At the same time, for every integer $0\le i\le k$, the word $u_i$ has to be square-free as well. 

The only square-free words over the alphabet $\{a,c\}$ are $a,c,ac,ca,aca$ and $cac$. Assume $u_m=aca$ for some integer $0\le m\le k$. Then $w=yacay'$ for some $y,y'\in\Sigma^*$. Observe that $w=yacay'\equiv_M yaacy'=ya^2cy'$, thus $ya^2cy'\in C_w$. However, this is impossible as every word in $C_w$ is square-free. Thus $u_i\neq aca$ for every integer $0\le i\le k$. Similarly, it can be shown that $u_i\neq cac$ for every integer $0\le i\le k$.

Assume $u_m=ac$ for some integer $0\le m\le k$.
\begin{case2}$0\le m\le k-2$.\\
Then $w=yacbu_{m+1}by'$ for some $y,y'\in\Sigma^*$. Since $u_{m+1}\not\in\{\lambda,u_{m}\}$, it follows that $u_{m+1}\in\{a,c,ca\}$. Observe that 
\begin{enumerate}
\item if $u_{m+1}=a$, then $w=yacbaby'\equiv_M ycababy'=yc(ab)^2y'$;
\item if $u_{m+1}=c$, then $w=yacbcby'=ya(cb)^2y'$;
\item if $u_{m+1}=ca$, then $w=yacbcaby'\equiv_M ycabcaby'=y(cab)^2y'$.
\end{enumerate}
Each case leads to a square word in $C_w$, thus a contradiction.
\end{case2}

\begin{case2}$k-1\le m\le k$.\\
Then $w=ybu_{m-1}bacy'$ for some $y,y'\in\Sigma^*$ because $k\ge 3$. Since $u_{m-1}\not\in\{\lambda,u_{m}\}$, it follows that $u_{m-1}\in\{a,c,ca\}$. From here, argue similarly as in Case 1. 
\end{case2}
Therefore, $u_i\neq ac$ for every integer $0\le i\le k$. Similarly, it can be shown that $u_i\neq ca$ for every integer $0\le i\le k$. Thus we conclude that $u_i\in\{a,c\}$ for all integers $1\le i\le k-1$ and $u_1,u_k\in\{\lambda,a,c\}$.
\end{proof}

\begin{lemma}\label{ConditionforSquareFreeClass}
Suppose $\Sigma=\{a<b<c\}$ and $w\in\Sigma^*$. If $|w|_b>4$, then there exists some word in $C_w$ that is not square-free.
\end{lemma}
\begin{proof}
We argue by contradiction. Assume that $|w|_b>4$ and every word in $C_w$ is square-free. Since $|w|_b>4$, it follows that $w=u_0bu_1bu_2bu_3bu_4\cdots u_{k-1}bu_k$ for some integer $k\ge 5$ and $u_i\in\{a,c\}^*\,(0\le i\le k)$. By Lemma~\ref{LemmaStructureSquareFree}, it holds that $u_i\in\{a,c\}$ for every integer $1\le i\le k-1$ such that $u_i\neq u_{i+1}$ for every integer $1\le i\le k-1$.

Therefore, if $u_1=a$, then $w=u_0babcbabc\cdots u_{k-1}bu_k=u_0(babc)^2\cdots u_{k-1}bu_k$. On the other hand, if $u_1=c$, then $w=u_0bcbabcba\cdots u_{k-1}bu_k=u_0(bcba)^2\cdots u_{k-1}bu_k$. In both cases, we have a contradiction as $w$ is square-free. Thus our conclusion holds.
\end{proof}

Let $\Sigma=\{a<b<c\}$. The following is an exhaustive list of every $M$\!-equivalence class over $\Sigma$ such that all words in it are square-free. The list can be verified by some simple analysis supported by Lemma~\ref{LemmaStructureSquareFree} and Lemma~\ref{ConditionforSquareFreeClass}.

\vspace{0.8em}$\begin{gathered}
\{a\}, \quad \{c\}, \quad \{ac,ca\}, \quad \{b\}, \quad \{ab\}, \quad \{ba\}, \quad \{cb\}, \quad \{bc\},\\
\{abc\}, \quad \{cba\}, \quad \{acb,cab\}, \quad \{bac,bca\},\\
\{acba,caba\}, \quad \{acbc,cabc\}, \quad \{abac,abca\}, \quad \{cbac,cbca\},\\
\{acbac,acbca,cabac,cabca\},\\
\{abcb\}, \quad \{cbab\}, \quad \{babc\}, \quad \{bcba\}, \quad \{bacb,bcab\},\\
\{bacba,bcaba\}, \quad \{bacbc,bcabc\}, \quad \{abacb,abcab\}, \quad \{cbacb,cbcab\},\\
\{abcba\}, \quad \{cbabc\}, \quad \{cbacbc,cbcabc\}, \quad \{abacba,abcaba\},\\
\{babcb\}, \quad \{bcbab\}, \quad \{abcbab\}, \quad \{cbabcb\}, \quad \{babcba\}, \quad \{bcbabc\},\\
\{babcbab\}, \quad \{bcbabcb\}.
\end{gathered}$

\vspace{0.8em}\noindent One can see by the above list that the following holds.
\begin{theorem}\label{TheoForConcl}
Suppose $\Sigma$ is an ordered alphabet with $|\Sigma|=3$. There are only finitely many $M$\!-equivalence classes over $\Sigma$ such that all words in each class are square-free. Furthermore, every such class is a $1$-equivalence class.
\end{theorem}

\begin{remark}
Note that in general, two distinct ternary square-free words that are \mbox{$M$\!-equivalent} need not be $1$-equivalent, for example the words $abcbabcacb$ and $bacabcbabc$. These two words form one of the \mbox{shortest-length} pairs of square-free ternary words that are \mbox{$M$\!-equivalent} but not $1$-equivalent (the only other such pair of words with length $10$ is $bcacbabcba$ and $cbabcbacab$).
\end{remark}

Next, we show that for an arbitrary ternary ordered alphabet, there are infinitely many pairs of square-free words that are $M$\!-equivalent. \!However, we first need the following notion and known result.

\begin{definition}
Suppose $\Sigma$ is an alphabet and $k$ is a positive integer. The \textit{k-spectrum} of a word $w\in\Sigma^*$ is the set $\{(u,|w|_u)\in \Sigma^*\times\mathbb{N}\,\,\,|\,\,\, |u|\le k\,\}$.
\end{definition}

The notion of $k$-spectrum was originally termed as $k$-deck in its introduction in \cite{MMSSS91}. Some examples of prominent works on $k$-spectrum are \cite{MMSSS91,jM00,DS03,aS05b}. 

\begin{theorem}\label{MorphPreserveSpec}\cite{wT16c}
Suppose $\Sigma,\Pi$ are alphabets, $\phi:\Sigma^*\rightarrow\Pi^*$ is a morphism, and $k$ is a positive integer. If two words $w,w'\in\Sigma^*$ have the same $k$-spectrum, then $\phi(w)$ and $\phi(w')$ have the same $k$-spectrum as well.
\end{theorem}

Suppose now $\Sigma$ is a ternary alphabet and $w,w'\in\Sigma^*$. By Remark~\ref{RemEntriesParMat}, one can see that if $w$ and $w'$ have the same 3-spectrum, then $w$ and $w'$ have the same Parikh matrix (i.e.~ $w$ and $w'$ are $M$\!-equivalent) with respect to any ordered alphabet with underlying alphabet $\Sigma$. Therefore, Theorem~\ref{MorphPreserveSpec} is desirably resourceful as it allows us to generate infinitely many pairs of ternary words that have the same $3$-spectrum via an arbitrary morphism. 

However, since our aim is to generate infinitely many pairs of $M$\!-equivalent ternary words that are square-free, the chosen morphism should preserve the square-freeness of the words throughout the (infinitely many) applications. Such morphisms are known as square-free morphisms and they are well-studied and presented in the literature (for example, \cite{jL57,mC82,jB84}).

\begin{example}\label{ExGenerateSFinfinite}
Suppose $\Sigma=\{a,b,c\}$ and consider the square-free morphism \mbox{$\phi:\Sigma^*\rightarrow\Sigma^*$} defined by:
\begin{equation*}
\begin{aligned}
\phi(a)=abcbacbcabcba; & &\phi(b)=bcacbacabcacb; & &\phi(c)=cabacbabcabac.
\end{aligned}
\end{equation*}
Let $w,w'\in\Sigma^*$ be words such that $w$ and $w'$ have the same 3-spectrum. For all integers $i>0$, define $w_i=\phi(w_{i-1})$ and $w'_i=\phi(w'_{i-1})$. Then for all integers $i\ge 0$, the words $w_i$ and $w'_i$ are both square-free and they have the same $3$-spectrum.
\end{example}

\begin{remark}
The morphism used in Example~\ref{ExGenerateSFinfinite} is in fact a uniform \mbox{square-free} morphism, meaning that under the morphism, all letters have images of the same length (in this case, it is 13). This morphism is due to Leech\cite{jL57}.
\end{remark}

Therefore, the final step is to find a pair of ternary square-free words having the same $3$-spectrum. Suprisingly, the shortest length of such ternary words is 18. The following exhausts all pairs of such words:
\begin{equation*}
\begin{aligned}
\{cabacbabcbacbcacba, abcacbcabcbabcabac\}\,\,\\
\{acbabcbacbcacbacab, bacabcacbcabcbabca\}\,\,\\
\{cbacabacbabcbacbca, acbcabcbabcabacabc\}\,\,\\
\{bcabacabcacbcabcba, abcbacbcacbacabacb\}\,\,\\
\{bcacbacabacbabcbac, cabcbabcabacabcacb\}\,\,\\
\{cbabcabacabcacbcab, bacbcacbacabacbabc\}.
\end{aligned}
\end{equation*}
By now, it is clear that the following result holds.

\begin{theorem}\label{sqfree3spec}
Suppose $\Sigma$ is an alphabet with $|\Sigma|=3$. There are infinitely many pairs of square-free words $w,w'\in\Sigma^*$ such that $w$ and $w'$ have the same $3$-spectrum. Furthermore, every such pair of words $w$ and $w'$ are $M$\!-equivalent with respect to any ordered alphabet with underlying alphabet $\Sigma$.
\end{theorem}

The relation \textit{ME}-equivalence is strictly stronger than $M$-equivalence in the sense that any two \textit{ME}-equivalent words are $M$\!-equivalent but not vice versa. The shortest pairs of square-free ternary words that are \textit{ME}-equivalent are of length 15 and as exhausted below:

\begin{equation*}
\begin{aligned}
\{cbacabacbabcacb, abcacbcabcbacab\}\,\,\\
\{bcacbabcabacabc, bacabcbacbcacba\}.
\end{aligned}
\end{equation*}

Finally, we end this section with a conjecture on the class of square-free ternary words with respect to $M$-unambiguity.

\begin{definition}
Suppose $\Sigma$ is an ordered alphabet with $|\Sigma|=3$. A word $w\in\Sigma^*$ is \textit{square-free\,-$M$-unambiguous} iff $w$ is not $M$\!-equivalent to any other distinct square-free word. 
\end{definition}

The table in Appendix~\ref{A2} shows that up until length 60, the proportion of square-free\,-$M$-unambiguous words eventually decreases steadily. Thus, we conjecture the following:
\begin{conjecture}
Suppose $\Sigma$ is an ordered alphabet with $|\Sigma|\ge 3$. Then
$$\lim\limits_{k\rightarrow\infty}\frac{|\{w\in\Sigma^*:w\text{ is square-free\,-$M$-unambiguous and }|w|=k\}|}{|\{w\in\Sigma^*:w\text{ is square-free and }|w|=k\}|}=0.$$
\end{conjecture}

\section{Conclusion}
The introduction of the notion of general prints of a word opens up a potential direction in studying repetition in words. It is interesting to see what other canonical ways of decompositions can be proposed to attain a general print of a word.

The behavior of the function $\zeta_r$ is intriguing. The monotonicity of the \mbox{function} is suggested by Appendix~\ref{A1} but it is yet to be conclusively \mbox{determined}. Some natural directions of research concerning the function $\zeta_r$ would be:
\begin{itemize}[leftmargin=2em]
\item to determine (for feasible range of $n$) the values $\zeta_r(n)$ for higher orders $r$, or estimate the values $\zeta_r(n)$ within bounds as tight as possible;
\item to study the growth rate of $\zeta_r$.
\end{itemize}

Finally, Lemma~\ref{ConditionforSquareFreeClass} implies that every square-free word $w\in\{a<b<c\}^*$ with $|w|_b>4$ is $M$\!-ambiguous. That is to say, there exists an upper bound on the length of $M$\!-unambiguous square-free words for the ternary alphabet (precisely, the upper bound is 7 as shown in the list before Theorem~\ref{TheoForConcl}). For our future work, we aim to generalize Lemma~\ref{ConditionforSquareFreeClass} to obtain the corresponding upper bounds for larger alphabets.

\section*{Acknowledgement}
The first and third authors gratefully acknowledge support for this research by a Research University Grant No.~1001/PMATHS/8011019 of Universiti Sains Malaysia.

\bibliographystyle{abbrv}

\appendix
\section{Values of $\zeta_2(n)$ for every integer \mbox{$1\le n\le 24$}}\label{A1}
\begin{tabular}{l l l}
N($n$)\!\!\!&=&\!\!\!$|\,\{w\in\Sigma^* \text{ with } |w|=n \text{ and } l_R(w)=\zeta_2(n)\}\,|$\\
P($n$)\!\!\!&=&\!\!\!$\dfrac{\text{Nu}(n)}{|\,\{w\in\Sigma^* \text{ with } |w|=n\}\,|}\times 100\%$ (rounded to two decimal places)
\end{tabular}
\\
\begin{table}[htbp]

    \begin{minipage}{.5\linewidth}
      
      \centering
        \begin{tabular}{|r|r|r|r|}
        \hline
        $n$ & \multicolumn{1}{c|}{$\zeta_2(n)$} & \multicolumn{1}{c|}{\text{N($n$)}} & \multicolumn{1}{c|}{P($n$)} \\ \hline
        1  & 0       & 2       & 100.00 \\ \hline
        2  & 1       & 2       & 50.00  \\ \hline
        3  & 1       & 6       & 75.00  \\ \hline
        4  & 1       & 16      & 100.00  \\ \hline
        5  & 2       & 8       & 25.00  \\ \hline
        6  & 2       & 24      & 37.50  \\ \hline
        7  & 3       & 2       & 1.56  \\ \hline
        8  & 3       & 16      & 6.25  \\ \hline
        9  & 3       & 64      & 12.50  \\ \hline
        10 & 3       & 178     & 17.38  \\ \hline
        11 & 4       & 10      & 0.49  \\ \hline
        12 & 4       & 48      & 1.17  \\ \hline
        13 & 4       & 180     & 2.20  \\ \hline
        14 & 4    & 552     & 3.37  \\ \hline
        15 & 4    & 1642    & 5.01  \\ \hline
        \end{tabular}
    \end{minipage}%
    \begin{minipage}{.5\linewidth}
      \centering
        
        \begin{tabular}{|r|r|r|r|}
        \hline
        $n$ & \multicolumn{1}{c|}{$\zeta_2(n)$} & \multicolumn{1}{c|}{\text{N($n$)}} & \multicolumn{1}{c|}{P($n$)} \\  \hline

        16 & 4    & 4410    & 6.73  \\ \hline
        17 & 4    & 11286   & 8.61  \\ \hline
        18 & 5    & 24      & 0.01  \\ \hline
        19 & 5    & 266     & 0.05  \\ \hline
        20 & 5    & 1314    & 0.13  \\ \hline
        21 & 5    & 4996    & 0.24  \\ \hline
        22 & 5    & 16134   & 0.38  \\ \hline
        23 & 5    & 47214   & 0.56  \\ \hline
        24 & 5    & 128846  & 0.77  \\ \hline
        25 & 5    & 333068  & 0.99  \\ \hline
        26 & 5    & 830620  & 1.24  \\ \hline
        27 & 5    & 2015582 & 1.50  \\ \hline
        28 & 5    & 4794990 & 1.79  \\ \hline
        29 & 5    & 11225526 & 2.09  \\ \hline
        30 & 6    & 18      & 0.00  \\ \hline
        \end{tabular}
    \end{minipage} 
\end{table}

\newpage

\section{Proportion of Square-free-\textit{M}-unambiguous Words}\label{A2}

\begin{tabular}{l l l}
Sq($k$) &= &Number of square-free words of length $k$\\
Un($k$) &= &Number of square-free\,-$M$-unambiguous words of length $k$\\
Pr($k$) &= &$\dfrac{\text{Sq}(k)}{\text{Un}(k)}\times 100\%$ (rounded to one decimal place)
\end{tabular}

\begin{table}[htbp]

    \begin{minipage}{.5\linewidth}
      
      \centering
        \begin{tabular}{|r|r|r|r|}
        \hline
        $k$ & \multicolumn{1}{c|}{\text{Sq($k$)}} & \multicolumn{1}{c|}{\text{Un($k$)}} & \multicolumn{1}{c|}{Pr($k$)} \\ \hline
        1  & 3        & 3       & 100.0 \\ \hline
        2  & 6        & 4       & 66.7  \\ \hline
        3  & 12       & 8       & 66.7  \\ \hline
        4  & 18       & 8       & 44.4  \\ \hline
        5  & 30       & 18      & 60.0  \\ \hline
        6  & 42       & 26      & 61.9  \\ \hline
        7  & 60       & 42      & 70.0  \\ \hline
        8  & 78       & 60      & 76.9  \\ \hline
        9  & 108      & 82      & 75.9  \\ \hline
        10 & 144      & 114     & 79.1  \\ \hline
        11 & 204      & 162     & 79.4  \\ \hline
        12 & 264      & 196     & 74.2  \\ \hline
        13 & 342      & 274     & 80.1  \\ \hline
        14 & 456      & 348     & 76.3  \\ \hline
        15 & 618      & 470     & 76.1  \\ \hline
        16 & 798      & 574     & 71.9  \\ \hline
        17 & 1044     & 780     & 74.7  \\ \hline
        18 & 1392     & 1004    & 72.1  \\ \hline
        19 & 1830     & 1296    & 70.8  \\ \hline
        20 & 2388     & 1650    & 69.1  \\ \hline
        21 & 3180     & 2232    & 70.1  \\ \hline
        22 & 4146     & 2848    & 68.7  \\ \hline
        23 & 5418     & 3670    & 67.7  \\ \hline
        24 & 7032     & 4818    & 68.5  \\ \hline
        25 & 9198     & 6242    & 67.9  \\ \hline
        26 & 11892    & 8024    & 67.5  \\ \hline
        27 & 15486    & 10308   & 66.5  \\ \hline
        28 & 20220    & 13222   & 65.4  \\ \hline
        29 & 26424    & 16850   & 63.8  \\ \hline
        30 & 34422    & 21578   & 62.7  \\ \hline
        \end{tabular}
    \end{minipage}%
    \begin{minipage}{.5\linewidth}
      \centering
        
        \begin{tabular}{|r|r|r|r|}
        \hline
        $k$ & \multicolumn{1}{c|}{\text{Sq($k$)}} & \multicolumn{1}{c|}{\text{Un($k$)}} & \multicolumn{1}{c|}{Pr($k$)} \\ \hline
        31 & 44862    & 27128   & 60.5  \\ \hline
        32 & 58446    & 33944   & 58.1  \\ \hline
        33 & 76122    & 42676   & 56.1  \\ \hline
        34 & 99276    & 53152   & 53.5  \\ \hline
        35 & 129516   & 66494   & 51.3  \\ \hline
        36 & 168546   & 82480   & 48.9  \\ \hline
        37 & 219516   & 102056  & 46.5  \\ \hline
        38 & 285750   & 125072  & 43.8  \\ \hline
        39 & 377204   & 153434  & 40.7  \\ \hline
        40 & 484446   & 186752  & 38.5  \\ \hline
        41 & 630666   & 226264  & 35.9  \\ \hline
        42 & 821154   & 272408  & 33.2  \\ \hline
        43 & 1069512  & 327468  & 30.6  \\ \hline
        44 & 1392270  & 390042  & 28.0  \\ \hline
        45 & 1812876  & 460248  & 25.4  \\ \hline
        46 & 2359710  & 541526  & 22.9  \\ \hline
        47 & 3072486  & 634254  & 20.6  \\ \hline
        48 & 4000002  & 741450  & 18.5  \\ \hline
        49 & 5207706  & 856702  & 16.5  \\ \hline
        50 & 6778926  & 989104  & 14.6  \\ \hline
        51 & 8824956  & 1147932 & 13.0  \\ \hline
        52 & 11488392 & 1313758 & 11.4  \\ \hline
        53 & 14956584 & 1497312 & 10.0  \\ \hline
        54 & 19470384 & 1711700 & 8.8   \\ \hline
        55 & 25346550 & 1953100 & 7.7   \\ \hline
        56 & 32996442 & 2213664 & 6.7   \\ \hline
        57 & 42957300 & 2485178 & 5.8   \\ \hline
        58 & 55921896 & 2834244 & 5.1   \\ \hline
        59 & 72798942 & 3192170 & 4.4   \\ \hline
        60 & 94766136 & 3571018 & 3.8   \\ \hline
        \end{tabular}
    \end{minipage} 
\end{table}

\end{document}